\documentclass[12pt]{article}
\usepackage{amsmath,amssymb,amsthm}
\numberwithin{equation}{section}
\newtheorem{theorem}{Theorem}[section]
\newtheorem{corollary}[theorem]{Corollary}
\newtheorem{lemma}[theorem]{Lemma}
\newtheorem{proposition}[theorem]{Proposition}
\theoremstyle{definition}
\newtheorem{definition}[theorem]{Definition}
\newtheorem{remark}[theorem]{Remark}

\newtheorem*{example}{Example}
\newcommand{\R}{\mathbb{R}}
\newcommand{\C}{\mathbb{C}}
\newcommand{\N}{\mathbb{N}}
\newcommand{\T}{\mathbb{T}}
\newcommand{\Z}{\mathbb{Z}}
\newcommand{\D}{\mathcal{D}}
\renewcommand{\L}{\mathcal{L}}
\newcommand{\M}{{\rm M}}
\newcommand{\MV}{{\rm MV}}
\newcommand{\Var}{{\rm Var\,}}
\newcommand{\const}{{\rm const}}
\newcommand{\supp}{\mathop{\rm supp}}
\newcommand{\meas}{\mathop{\rm meas}}
\newcommand{\pr}{{\rm pr}}
\renewcommand{\limsup}{\mathop{\overline{\lim}}}

\newcommand{\esslim}{\mathop{\rm ess\,lim}}

\newcommand{\sign}{\mathop{\rm sign}}

\renewcommand{\div}{\mathrel{\rm div}}

\def\Xint#1{\mathchoice
{\XXint\displaystyle\textstyle{#1}}%
{\XXint\textstyle\scriptstyle{#1}}%
{\XXint\scriptstyle\scriptscriptstyle{#1}}%
{\XXint\scriptscriptstyle\scriptscriptstyle{#1}}%
\!\int}
\def\XXint#1#2#3{{\setbox0=\hbox{$#1{#2#3}{\int}$ }
\vcenter{\hbox{$#2#3$ }}\kern-.57\wd0}}

\def\dashint{\Xint-}

\title{On a condition of strong precompactness and the decay of periodic entropy solutions to scalar conservation laws}

\author{Evgeny Yu. Panov\footnote{Novgorod State University, e-mail: \textbf{Eugeny.Panov@novsu.ru}}}
\date{}
\begin{document}
\maketitle

\begin{abstract}
We propose a new sufficient non-degeneracy condition for the strong precompactness of bounded sequences satisfying the nonlinear first-order differential constraints. This result is applied to establish the decay property for periodic entropy solutions to multidimensional scalar conservation laws.
\end{abstract}

\section{Introduction}

Let $\Omega$ be an open domain in $\R^n$.
We consider the sequence $u_k(x)$, $k\in\N$, bounded in $L^\infty(\Omega)$,
which converges weakly-$*$ in $L^\infty(\Omega)$ to some function $u(x)$:
$u_k\mathop{\rightharpoonup}\limits_{k\to\infty} u$. Now let $\varphi(x,u)\in L^2_{loc}(\Omega,C(\R,\R^n))$
be a Caratheodory vector-function (i.e. it is continuous with
respect to $u$ and measurable with respect to $x$) such that the functions
\begin{equation}\label{1}
\alpha_M(x)=\max\limits_{|u|\le M} |\varphi(x,u)|\in
L^2_{loc}(\Omega) \quad \forall M>0
\end{equation}
(here and below $|\cdot|$ stands for the Euclidean norm of a
finite-dimensional vector).

By
$\theta(\lambda)$ we shall denote the Heaviside function:
$$
\theta(\lambda)=\left\{\begin{array}{ll}
1,~&\lambda>0, \\
0,~&\lambda\le 0.
\end{array}\right.
$$
Suppose that for every $p\in\R$ the sequence of distributions
\begin{equation}\label{2}
\div_x [\theta(u_k-p)(\varphi(x,u_k)-\varphi(x,p))] \ \mbox{ is precompact in } W_{d,loc}^{-1}(\Omega)
\end{equation}
for some $d>1$. Recall that $W_{d,loc}^{-1}(\Omega)$ is a locally convex space of
 distributions $u(x)$  such  that  $uf(x)$
belongs to the Sobolev space $W_d^{-1}$ for all $f(x)\in
C_0^\infty(\Omega)$.
 The topology in $W_{d,loc}^{-1}(\Omega)$ is generated
by the family of semi-norms $u\to\|uf\|_{W_d^{-1}}$, $f(x)\in
C_0^\infty(\Omega)$.

If the distributions $\div_x\varphi(x,k)$ are locally finite measures on $\Omega$ for all $k\in\R$, then the notion
of entropy solutions (in Kruzhkov's sense) of the equation
\begin{equation}\label{eq1}
\div\varphi(x,u)+\psi(x,u)=0
\end{equation}
(with a Caratheodory source function $\psi(x,u)\in L^1_{loc}(\Omega,C(\R))$) is defined, see \cite{PaARMA} and \cite{PaJMS} (in the latter paper the more general ultra-parabolic equations are studied).
As was shown in \cite{PaJMS}, assumption (\ref{2}) is always satisfied for bounded sequences of entropy solutions
of (\ref{eq1}).

Our first result is the following strong precompactness property.

\begin{theorem}\label{th1}
Suppose that for almost every $x\in\Omega$ and all $\xi\in\R^n$, $\xi\not=0$ the function
$\lambda\to\xi\cdot\varphi(x,\lambda)$ is not constant in any vicinity of the point $u(x)$ (here and in the sequel ``$\cdot$'' denotes the inner product in $\R^n$).
Then $u_k(x)\mathop{\to}\limits_{k\to\infty} u(x)$ in $L^1_{loc}(\Omega)$ (strongly).
\end{theorem}

Theorem~\ref{th1} extends the results of \cite{PaARMA}, where the strong precompactness property was established
under the more restrictive non-degeneracy condition:  {\it for almost every $x\in\Omega$ and all $\xi\in\R^n$, $\xi\not=0$ the function $\lambda\to\xi\cdot\varphi(x,\lambda)$ is not constant on nonempty intervals. }

The proof of Theorem~\ref{th1} is based on a new localization principle for H-measure (with ``continuous'' indexes)   corresponding to the sequence $u_k$, see Theorem~\ref{th3} and its Corollary~\ref{cor2} below.

Using this theorem and results of \cite{PaAIHP},
we will also derive the more precise criterion for the decay of periodic entropy solutions of scalar conservation laws
\begin{equation}\label{cl}
u_t+\div_x\varphi(u)=0,
\end{equation}
$u=u(t,x)$, $(t,x)\in\Pi=(0,+\infty)\times\R^n$. The flux vector $\varphi(u)=(\varphi_1(u),\ldots,\varphi_n(u))$
is supposed to be merely continuous: $\varphi(u)\in C(\R,\R^n)$.
Recall the definition of entropy solution to equation (\ref{cl}) in the Kruzhkov sense \cite{Kr}.

\begin{definition}\label{def1}
A bounded measurable function $u=u(t,x)\in L^\infty(\Pi)$ is called an entropy solution (e.s. for
short) of (\ref{cl}) if for all $k\in\R$
\begin{equation}\label{3}
|u-k|_t+\div_x[\sign(u-k)(\varphi(u)-\varphi(k))]\le 0
\end{equation}
in the sense of distributions on $\Pi$ (in $\D'(\Pi)$).
\end{definition}
As usual, condition (\ref{3}) means that for all non-negative test functions $f=f(t,x)\in C_0^1(\Pi)$
$$
\int_\Pi [|u-k|f_t+\sign(u-k)(\varphi(u)-\varphi(k))\cdot\nabla_xf]dtdx\ge 0.
$$
As was shown in \cite{Pa6} (see also \cite{Pa7}), an e.s. $u(t,x)$ always admits a strong trace
$u_0=u_0(x)\in L^\infty(\R^n)$ on the initial hyperspace $t=0$ in the sense of relation
\begin{equation}\label{4}
\esslim_{t\to 0} u(t,\cdot)=u_0 \ \mbox{ in } L^1_{loc}(\R^n),
\end{equation}
that is, $u(t,x)$ is an e.s. to the Cauchy problem for equation (\ref{cl}) with initial data
\begin{equation}\label{ini}
u(0,x)=u_0(x).
\end{equation}
\begin{remark}\label{rem1}
It was also established in \cite[Corollary~7.1]{Pa6} that, after possible correction on a set of null
measure, an e.s. $u(t,x)$ is continuous on $[0,+\infty)$ as a map $t\mapsto u(t,\cdot)$ into
$L^1_{loc}(\R^n)$. In the sequel we will always assume that this property is satisfied.
\end{remark}
Suppose that the initial function $u_0$ is periodic with a lattice of periods $L$, i.e., $u_0(x+e)=u_0(x)$ a.e. on $\R^n$ for every $e\in L$ (we will call such functions $L$-periodic). Denote by $\T^n=\R^n/L$ the corresponding $n$-dimensional torus, and by $L'$ the dual lattice $L'=\{ \ \xi\in\R^n \ | \ \xi\cdot x\in\Z \ \forall x\in L \ \}$.
In the case under consideration when the flux vector is merely continuous the property of finite speed of propagation for initial perturbation may be violated, which, in the multidimensional situation $n>1$, may even lead to the nonuniqueness of e.s. to Cauchy problem (\ref{cl}), (\ref{ini}), see examples in \cite{KrPa1,KrPa2}. But for a periodic initial function $u_0(x)$, an e.s. $u(t,x)$ of (\ref{cl}), (\ref{ini}) is unique (in the class of all e.s., not necessarily periodic) and space-periodic, the proof can be found in \cite{PaMax1}. It is also shown in \cite{PaMax1} that the mean value of e.s. over the period does not depend on time:
\begin{equation}\label{mass}
\int_{\T^n}u(t,x)dx=I\doteq\int_{\T^n}u_0(x)dx,
\end{equation}
where $dx$ is the normalized Lebesgue measure on $\T^n$.
The following theorem generalizes the previous results of \cite{Daferm,PaAIHP}.

\begin{theorem}\label{th2}
Suppose that
\begin{eqnarray}\label{ND2}
\forall\xi\in L', \xi\not=0 \ \mbox{ the function } u\to\xi\cdot\varphi(u) \nonumber\\ \mbox{ is not affine on any vicinity of } I.
\end{eqnarray}
Then
\begin{equation}\label{dec}
\lim_{t\to +\infty} u(t,\cdot)=I=\int_{\T^n} u_0(x)dx \ \mbox{ in } L^1(\T^n).
\end{equation}
Moreover condition (\ref{ND2}) is necessary and sufficient for the decay property (\ref{dec}).
\end{theorem}

In the case $\varphi(u)\in C^2(\R,\R^n)$ Theorem~\ref{th2} was proved in \cite{Daferm}. As was noticed in
\cite[Remark~2.1]{Daferm}, decay property (\ref{dec}) holds under the the weaker regularity requirement $\varphi(u)\in C^1(\R,\R^n)$ but under the more restrictive assumption that for each $\xi\in L'$ \ $I$ is not an interior point of the closure of the union of all open intervals, over which the function $\xi\cdot\varphi'(u)$
is constant. Let us demonstrate that condition (\ref{ND2}) is less restrictive than this assumption even in the case  $\varphi(u)\in C^1(\R,\R^n)$. Suppose that $n=1$, $\varphi(u)\in C^1(\R)$ is a primitive of the Cantor function, so that $\varphi'(u)$ is increasing, continuous, and maximal intervals, over which it remains constant, are exactly the connected component of the complement $\R\setminus K$ of the Cantor set $K\subset [0,1]$. Since $K$ has the empty interior the assumption of \cite{Daferm} is never satisfied while (\ref{ND2}) holds for each $I\in K$.

\section{Preliminaries}\label{sec1}
We need the concept of measure valued functions (Young measures). Recall (see \cite{Di,Ta1}) that a measure-valued function on $\Omega$ is a weakly
measurable map $x\mapsto \nu_x$ of $\Omega$ into the space $\operatorname{Prob}_0(\R)$ of probability
Borel measures with compact support in~$\R$.

The weak measurability of $\nu_x$ means that for each continuous function $g(\lambda)$ the
function $x\to\langle\nu_x,g(\lambda)\rangle\doteq\int g(\lambda)d\nu_x(\lambda)$ is measurable on~$\Omega$.

A  measure-valued function $\nu_x$  is said to be bounded if there
exists $M>0$ such  that $\supp\nu_x\subset[-M,M]$  for almost all
$x\in\Omega.$

Measure-valued functions of the kind
$\nu_x(\lambda)=\delta(\lambda-u(x))$, where $u(x)\in L^\infty(\Omega)$ and
$\delta(\lambda-u^*)$ is the Dirac measure at $u^*\in\R$, are called {\it regular}. We identify these
measure-valued functions and the corresponding functions $u(x)$, so that there is a natural
embedding of $L^\infty(\Omega)$ into the set $\MV(\Omega)$ of bounded measure-valued functions on~$\Omega$.

Measure-valued functions naturally arise as weak limits of bounded sequences in $L^\infty(\Pi)$ in the
sense of the following theorem by L.~Tartar \cite{Ta1}.

\begin{theorem}\label{thT}
Let $u_k(x)\in L^\infty(\Omega)$, $k\in\N$, be a bounded sequence. Then there exist a subsequence (we
keep the notation $u_k(x)$ for this subsequence) and a bounded measure valued function $\nu_x\in\MV(\Omega)$
such that\begin{equation} \label{pr2} \forall g(\lambda)\in C(\R) \quad g(u_k)
\mathop{\to}_{k\to\infty}\langle\nu_x,g(\lambda)\rangle \quad\text{weakly-\/$*$ in } L^\infty(\Omega).
\end{equation}
Besides, $\nu_x$ is regular, i.e., $\nu_x(\lambda)=\delta(\lambda-u(x))$ if and only if $u_k(x)
\mathop{\to}\limits_{k\to\infty} u(x)$ in $L^1_{loc}(\Omega)$ (strongly).
\end{theorem}

We will essentially use in the sequel the variant of H-measures with ``continuous indexes'' introduced in
\cite{Pa3}. This variant extends the original concept of H-measure invented by L.~Tartar \cite{Ta2} and P.~Ger\'ard \cite{Ger} and it appears to be a powerful tool in nonlinear analysis.

Suppose $u_k(x)$ is a bounded sequence in $L^\infty(\Omega)$. Passing to a subsequence if necessary, we can suppose that this sequence converges
to a bounded measure valued function $\nu_x\in\MV(\Omega)$ in the sense of relation (\ref{pr2}). We introduce
the measures $\gamma_x^k(\lambda)= \delta(\lambda-u_k(x))-\nu_x(\lambda)$ and the corresponding
distribution functions $U_k(x,p)=\gamma_x^k((p,+\infty))$, $u_0(x,p)=\nu_x((p,+\infty))$ on
$\Omega\times\R$. Observe that $U_k(x,p), u_0(x,p)\in L^\infty(\Omega)$ for all $p\in\R$, see
\cite[Lemma 2]{Pa3}. We define the set
$$
E=E(\nu_x)=\left\{ \ p_0\in\R \ \mid \
u_0(x,p)\mathop{\to}\limits_{p\to p_0}\, u_0(x,p_0) \ \mbox{ in }
L_{loc}^1(\Omega) \ \right\}.
$$
As was shown in \cite[Lemma 4]{Pa3}, the complement $\R\setminus E$ is at most countable and if
$p\in E$ then $U_k(x,p)\mathop{\rightharpoonup}\limits_{k\to\infty}\, 0$ weakly-$*$ in
$L^\infty(\Omega)$.

Let $F(u)(\xi),$ $\xi\in\R^n,$
be the Fourier transform of a function $u(x)\in L^2(\R^n)$,
$S=S^{n-1}= \{ \ \xi\in\R^n \ \mid \ |\xi|=1 \ \}$
 be the unit sphere in $\R^n$. Denote by
 $u\to\overline{u}$,
$u\in\C$ the complex conjugation.

The next result was established in
\cite[Theorem~3]{Pa3}, \cite[Proposition~2, Lemma~2]{Pa5}.

\begin{proposition}\label{pro2}
(i) There exists a family of locally finite complex Borel measures $\left\{\mu^{pq}\right\}_{p,q\in E}$
in $\Omega\times S$ and a subsequence $U_r(x,p)=U_{k_r}(x,p)$ such that for all
$\Phi_1(x),\Phi_2(x)\in C_0(\Omega)$ and $\psi(\xi)\in C(S)$
\begin{equation}\label{Hm}
\langle\mu^{pq},\Phi_1(x)\overline{\Phi_2(x)}\psi(\xi)\rangle= \lim\limits_{r\to\infty}\int_{\R^n}
F(\Phi_1U_r(\cdot,p))(\xi)\overline{F(\Phi_2U_r(\cdot,q))(\xi)} \psi\left(\frac{\xi}{|\xi|}\right)d\xi;
\end{equation}


(ii) For any $p_1,\ldots,p_l\in E$
 the matrix $\{\mu^{p_ip_j}\}_{i,j=1}^l$ is Hermitian and nonnegative definite, that is,
for all $\zeta_1,\ldots,\zeta_l\in\C$ the measure
$$\sum_{i,j=1}^l \zeta_i\overline{\zeta_j}\mu^{p_ip_j}\ge
0.$$
\end{proposition}

We call the family of measures
$\left\{\mu^{pq}\right\}_{p,q\in E}$ the H-measure
corresponding to the subsequence $u_r(x)=u_{k_r}(x)$.

As was demonstrated in \cite{Pa3}, the H-measure $\mu^{pq}=0$ for all $p,q\in E$ if and only if the
subsequence $u_r(x)$ converges as $r\to\infty$ strongly (in $L^1_{loc}(\Omega)$).

Since $|U_k(x,p)|\le 1$, it readily follows from (\ref{Hm}) and Plancherel's equality that $\pr_\Omega|\mu^{pq}|\le\meas$ for $p,q\in E$, where $\meas$ is the Lebesgue measure on $\Omega$, and by $|\mu|$ we denote the variation of a Borel measure $\mu$ (this is the minimal of nonnegative Borel measures $\nu$ such that $|\mu(A)|\le\nu(A)$ for all Borel sets $A$).
This implies the representation $\mu^{pq}=\mu^{pq}_xdx$ (the disintegration of H-measures).
More exactly, choose a countable dense subset $D\subset E$. The following statement was proved in \cite[Proposition~3]{Pa5}, see also \cite[Proposition~3]{PaARMA}.

\begin{proposition}\label{pro3}
There exists a family of complex finite
Borel  measures $\mu^{pq}_x\in \M(S)$
 in the sphere $S$ with $p,q\in D$, $x\in\Omega'$,
where $\Omega'$ is a subset  of $\Omega$ of full measure, such that
$\mu^{pq}=\mu^{pq}_xdx$, that is, for all $\Phi(x,\xi)\in
C_0(\Omega\times S)$ the function
$$
x\to\langle\mu^{pq}_x(\xi),\Phi(x,\xi)\rangle =\int_S \Phi(x,\xi)
d\mu^{pq}_x(\xi)
$$
is Lebesgue-measurable on $\Omega$,  bounded, and
$$
\langle\mu^{pq},\Phi(x,\xi)\rangle =\int_\Omega
\langle\mu^{pq}_x(\xi),\Phi(x,\xi)\rangle dx.
$$
Moreover, for $p,p',q\in D$, $p'>p$
\begin{equation}
\label{7} \Var\mu^{pq}_x\doteq|\mu^{pq}_x|(S)\le 1 \ \mbox{ and } \
\Var(\mu^{p'q}_x-\mu^{pq}_x)\le
2\left(\nu_x((p,p'))\right)^{1/2}.
\end{equation}
\end{proposition}
We choose a non-negative function $K(x)\in C_0^\infty (\R^n)$
 with support in the unit ball such that $\int K(x)dx=1$ and set $K_m(x)=m^n K(mx)$ for
$m\in\N$.
 Clearly, the sequence of $K_m$ converges
in  $\D'(\R^n)$  to the Dirac $\delta$-function (~that is, this
sequence is an approximate unity~). We define $\Phi_m(x)=(K_m(x))^{1/2}$. As was shown in \cite[Remark~4]{Pa5}
(see also \cite[Remark~2(b)]{PaARMA}~), the measures $\mu^{pq}_x$ can be explicitly represented by the relation
\begin{eqnarray}\label{repr}
\Phi(x)\langle\mu_x^{pq},\psi(\xi)\rangle=\lim_{m\to\infty}\langle\mu_x^{pq}(y,\xi),\Phi(y)K_m(x-y)\psi(\xi)\rangle=\nonumber\\
\lim\limits_{m\to\infty} \lim\limits_{r\to\infty} \int_{\R^n}
F(\Phi\Phi_mU_r(\cdot,p))(\xi)\overline{F(\Phi_mU_r(\cdot,q))(\xi)}
\psi\left(\frac{\xi}{|\xi|}\right)d\xi
\end{eqnarray}
for all $\psi(\xi)\in C(S)$,  where
$\Phi\Phi_mU_r^p(y)=\Phi(y)\Phi_m(x-y)U_r(y,p)$, $\Phi_m
U_r^q(y)=\Phi_m(x-y) U_r(y,q)$, and $\Phi(y)\in L^2_{loc}(\Omega)$ be an arbitrary function such that $x$ is its Lebesgue point.

From this representation (with $\Phi\equiv 1$) and Proposition~\ref{pro2}(ii) it follows that
for all $p_1,\ldots,p_l\in D$, $x\in\Omega'$, $\zeta_1,\ldots,\zeta_l\in\C$ the measure
\begin{equation}\label{pd1}
\mu=\sum_{i,j=1}^l \zeta_i\overline{\zeta_j}\mu^{p_ip_j}_x\ge
0.
\end{equation}
Indeed, for every nonnegative $\psi(\xi)\in C(S)$
$$
<\mu(\xi),\psi(\xi)>=\lim_{m\to\infty}\left\langle\sum_{i,j=1}^l \zeta_i\overline{\zeta_j}\mu^{p_ip_j}(y,\xi),K_m(x-y)\psi(\xi)\right\rangle\ge 0.
$$
This, in particular implies, that $\mu^{pp}_x\ge 0$, $\mu^{qp}_x=\overline{\mu^{pq}_x}$,
and for every Borel set $A\subset S$
\begin{equation}\label{pd2}
|\mu^{pq}_x|(A)\le \left(\mu^{pp}_x(A)\mu^{qq}_x(A)\right)^{1/2}
\end{equation}
(~see \cite{Pa5,PaARMA}~).
For completeness we provide below the simple proof of (\ref{pd2}).
In view of (\ref{pd1}) (with l=2)
the matrix $M=\left(\begin{array}{cc} \mu^{pp}_x(A) & \mu^{pq}_x(A) \\ \mu^{qp}_x(A) & \mu^{qq}_x(A)
\end{array}\right)$ is Hermitian and nonnegative definite. Therefore,
$$
\mu^{pp}_x(A)\mu^{qq}_x(A)-|\mu^{pq}_x(A)|^2=\mu^{pp}_x(A)\mu^{qq}_x(A)-\mu^{pq}_x(A)\mu^{qp}_x(A)=\det M\ge 0.
$$
By Young's inequality for any positive constant $c$ and all Borel sets $A\subset S$
$$
|\mu^{pq}_x(A)| \le \left(\mu^{pp}_x(A)\mu^{qq}_x(A)\right)^{1/2}\le \frac{c}{2}\mu^{pp}_x(A)+\frac{1}{2c}\mu^{qq}_x(A).
$$
Since $\mu=\displaystyle\frac{c}{2}\mu^{pp}_x+\frac{1}{2c}\mu^{qq}_x$ is nonnegative Borel measure, it follows from this inequality that the variation $|\mu^{pq}_x|\le\mu$. This implies that
\begin{equation}\label{pd3}
|\mu^{pq}_x|(A)\le \frac{c}{2}\mu^{pp}_x(A)+\frac{1}{2c}\mu^{qq}_x(A) \quad \forall c>0.
\end{equation}
It is easily computed that
$$
\inf_{c>0}\left(\frac{c}{2}\mu^{pp}_x(A)+\frac{1}{2c}\mu^{qq}_x(A)\right)=\left(\mu^{pp}_x(A)\mu^{qq}_x(A)\right)^{1/2}
$$
and (\ref{pd2}) follows from (\ref{pd3}).

\section{Localization principles and the strong precompactness property}
\begin{lemma}\label{lem1}
For each $p,q\in\R$, $x\in\Omega'$ there exist one-sided limits in the space $\M(S)$ of finite Borel measures on $S$ (with the standard norm $\Var\mu$):
\begin{eqnarray*}
\mu^{p'q'}_x\to\mu^{pq+}_x \ \mbox{ as } (p',q')\to (p,q), \quad p',q'\in D, p'>p, q'>q, \\
\mu^{p'q'}_x\to\mu^{pq-}_x \ \mbox{ as } (p',q')\to (p,q), \quad p',q'\in D, p'<p, q'<q.
\end{eqnarray*}
Moreover, $\Var\mu^{pq\pm}\le 1$ and for every Borel set $A\subset S$ and each $p_i\in\R$, $i=1,\ldots,l$ the matrices $\{\mu^{p_ip_j\pm}_x(A)\}_{i,j=1}^l$ are Hermitian and nonnegative definite, that is, the measures
\begin{equation}\label{pos1}
\sum_{i,j=1}^l \zeta_i\overline{\zeta_j}\mu^{p_ip_j\pm}_x\ge 0
\end{equation}
for all complex $\zeta_i\in\C$, $i=1,\ldots,l$.
\end{lemma}

\begin{proof}
Let $x\in\Omega'$, $p,q\in\R$, $p_1,q_1,p_2,q_2\in D$, $p_2>p_1>p$, $q_2>q_1>q$. Then, in view of (\ref{7}) and the equality $\mu^{qp}_x=\overline{\mu^{pq}_x}$,
\begin{eqnarray*}
\Var(\mu^{p_2q_2}_x-\mu^{p_1q_1}_x)\le 2\nu_x((p_1,p_2))+2\nu_x((q_1,q_2))\le \\ 2\nu_x((p,p_2))+2\nu_x((q,q_2))\mathop{\to}_{(p_2,q_2)\to (p,q)} 0.
\end{eqnarray*}
By the Cauchy criterion, this implies that there exists a limit $\mu^{pq+}_x$ in $\M(S)$ as $(p',q')\to (p,q)$,
$p',q'\in D$, $p'>p$, $q'>q$. Similarly, for each $p_1,q_1,p_2,q_2\in D$ such that $p_2<p_1<p$, $q_2<q_1<q$
\begin{eqnarray*}
\Var(\mu^{p_2q_2}_x-\mu^{p_1q_1}_x)\le 2\nu_x((p_2,p_1))+2\nu_x((q_2,q_1))\le \\ 2\nu_x((p_2,p))+2\nu_x((q_2,q))\mathop{\to}_{(p_2,q_2)\to (p,q)} 0,
\end{eqnarray*}
which implies existence of a left-sided limit $\mu^{pq-}_x$ in $\M(S)$ as $(p',q')\to (p,q)$,
$p',q'\in D$, $p'<p$, $q'<q$.
By Proposition~\ref{pro3} \ $\Var\mu^{p'q'}_x\le 1$, which implies in the limits as $p'\to p\pm$, $q'\to q\pm$ that
$\Var\mu^{pq\pm}_x\le 1$. Finally,  for every $p_i'\in D$, $\zeta_i\in\C$, $i=1,\ldots,l$ the measures
$$
\sum_{i,j=1}^l \zeta_i\overline{\zeta_j}\mu^{p'_ip'_j\pm}_x\ge 0.
$$
In the limits as $p'_i\to p_i\pm$ this implies (\ref{pos1}).
\end{proof}

\begin{corollary}\label{cor1}
Let $p,q\in\R$, $x\in\Omega'$. Then for every Borel set $A\subset S$
\begin{equation}\label{po1}
|\mu^{pq+}_x|(A)\le\left(\mu^{pp+}_x(A)\mu^{qq+}_x(A)\right)^{1/2}, \ |\mu^{pq-}_x|(A)\le\left(\mu^{pp-}_x(A)\mu^{qq-}_x(A)\right)^{1/2}.
\end{equation}
\end{corollary}
\begin{proof}
Relations (\ref{po1}) follow from (\ref{pos1}) in the same way as in the proof of inequality (\ref{pd2}) above.
\end{proof}
\begin{remark}\label{rem2}
By continuity of $\mu^{pq}_x$ with respect to variables $p,q\in D$, we see that for
$p\in D$
$$
\mu^{pq\pm}_x=\lim_{q'\to q\pm}\lim_{p'\to p\pm}\mu^{p'q'}_x=\lim_{q'\to q\pm}\mu^{pq'}_x \ \mbox{ in } \M(S).
$$
Analogously, if $q\in D$, then
$$
\mu^{pq\pm}_x=\lim_{p'\to p\pm}\mu^{p'q}_x \ \mbox{ in } \M(S).
$$
If the both indices $p,q\in D$, then evidently $\mu^{pq\pm}_x=\mu^{pq}_x$.
\end{remark}

Now we suppose that $f(y,\lambda)\in L^2_{loc}(\Omega,C(\R,\R^n))$ is a Caratheodory vector-function on
$\Omega\times\R$. In particular,
\begin{equation}
\label{growth} \forall M>0 \quad
\|f(x,\cdot)\|_{M,\infty}=\max\limits_{|\lambda|\le
M}|f(x,\lambda)|=\alpha_M(x)\in L^2_{loc}(\Omega).
\end{equation}
Since the space $C(\R,\R^n)$ is separable with respect to the
standard locally convex topology generated by seminorms
$\|\cdot\|_{M,\infty}$, then, by the Pettis theorem (see \cite{HF},
Chapter~3), the map $x\to F(x)=f(x,\cdot)\in C(\R,\R^n)$ is strongly
measurable and in view of estimate (\ref{growth}) we see that
$|F(x)|^2\in
L^1_{loc}(\Omega,C(\R))$. In particular (see \cite{HF}, Chapter~3),
the set $\Omega_f$ of common Lebesgue points of the maps
$F(x),|F(x)|^2$ has full measure. As was demonstrated in \cite{PaARMA}, for $x\in\Omega_f$
\begin{equation}
\label{leb} \lim_{m\to\infty}\int
K_m(x-y)\|F(x)-F(y)\|_{M,\infty}^2dy=0 \quad \forall M>0.
\end{equation}
Clearly, each $x\in\Omega_f$ is a common Lebesgue point of all functions
$x\to f(x,\lambda)$, $\lambda\in\R$. Let
$\Omega''=\Omega'\cap\Omega_f$, $\gamma_x^r(\lambda)=\delta(\lambda-u_r(x))-\nu_x(\lambda)$.

Suppose that $x\in\Omega''$, $p\in\R$, $H_+,H_-$ are the minimal linear subspaces of $\R^n$, containing supports of the measures $\mu^{pp+}_x$, $\mu^{pp-}_x$, respectively. We fix $q\in D$ and introduce for $p'\in D$ the function
\begin{equation}\label{Ir}
I_r(y,p')=\int f(y,\lambda)(\theta(\lambda-p')-\theta(\lambda-q))d\gamma_y^r(\lambda)\in L^2_{loc}(\Omega).
\end{equation}

\begin{proposition}\label{pro4}
Assume that $q>p$ and $f(x,\lambda)\in H_+^\bot$ for all $\lambda\in\R$. Then
\begin{equation}\label{lim+}
\lim_{p'\to p+}\lim_{m\to\infty} \lim_{r\to\infty}
\int_{\R^n}\frac{\xi}{|\xi|}\cdot
F(\Phi_mI_r(\cdot,p'))(\xi)\overline{F(\Phi_m
U_r(\cdot,p'))(\xi)}\psi\left(\frac{\xi}{|\xi|}\right)d\xi=0
\end{equation}
for all $\psi(\xi)\in C(S)$.
Analogously, if $q<p$ and $f(x,\lambda)\in H_-^\bot$ $\forall\lambda\in\R$, then $\forall \psi(\xi)\in C(S)$
\begin{equation}\label{lim-}
\lim_{p'\to p-}\lim_{m\to\infty} \lim_{r\to\infty}
\int_{\R^n}\frac{\xi}{|\xi|}\cdot
F(\Phi_mI_r(\cdot,p'))(\xi)\overline{F(\Phi_m
U_r(\cdot,p'))(\xi)}\psi\left(\frac{\xi}{|\xi|}\right)d\xi=0.
\end{equation}
 Here $\Phi_m=\Phi_m(x-y)=\sqrt{K_m(x-y)}$ and $I_r(y,p'), U_r(y,p')$
 are functions of the variable $y\in\Omega$.
\end{proposition}

\begin{proof}  Note that starting from some index $m$ the supports of
the functions $\Phi_m(x-y)$ lie in  some  compact subset $B$ of $\Omega$.
 Without loss of generality  we  can
assume  that  $\supp\Phi_m\subset B$
  for all $m\in\N$. Let
$$
\tilde I_r(y,p')=\int f(x,\lambda)(\theta(\lambda-p')-\theta(\lambda-q))d\gamma_y^r(\lambda)\in L^2_{loc}(\Omega),
$$
$M=\sup\limits_{r\in\N}\|u_r\|_\infty$. Then $\supp \gamma^r_y\subset [-M,M]$, and
  $$|I_r(y,p')-\tilde I_r(y,p')|\le \int
  |f(y,\lambda)-f(x,\lambda)|d |\gamma^r_y|(\lambda)\le
  2\|F(y)-F(x)\|_{M,\infty}.$$
By Plancherel's identity
\begin{eqnarray*}
\left|\int_{\R^n}\frac{\xi}{|\xi|}\cdot
F(\Phi_mI_r(\cdot,p'))(\xi) \overline{F(\Phi_m
U_r(\cdot,p'))(\xi)}\psi\left(\frac{\xi}{|\xi|}\right)d\xi-\right. \\
\left.
\int_{\R^n}\frac{\xi}{|\xi|}\cdot
F(\Phi_m\tilde I_r(\cdot,p'))(\xi) \overline{F(\Phi_m
U_r(\cdot,p'))(\xi)}\psi\left(\frac{\xi}{|\xi|}\right)d\xi\right|= \\
\left|\int_{\R^n}\frac{\xi}{|\xi|}\cdot
F(\Phi_m(I_r(\cdot,p')-\tilde I_r(\cdot,p')))(\xi) \overline{F(\Phi_m
U_r(\cdot,p'))(\xi)}\psi\left(\frac{\xi}{|\xi|}\right)d\xi\right|
\le
\\ \|\psi\|_\infty\|\Phi_m(I_r(\cdot,p')-\tilde I_r(\cdot,p'))\|_2\|\Phi_m
U_r(\cdot,p')\|_2\le \\  \|\psi\|_\infty\|\Phi_m(I_r(\cdot,p')-\tilde I_r(\cdot,p'))\|_2\le \\
2\|\psi\|_\infty\left(\int K_m(x-y)\|F(y)-F(x)\|_{M,\infty}^2 dy\right)^{1/2}.
\end{eqnarray*}
Here we take account  of the equality
 $$
\|\Phi_m\|_2=\left(\int_\Omega K_m(x-y)dy\right)^{1/2}=1.
$$
From the above estimate and (\ref{leb}) it follows that
\begin{eqnarray}
\label{red} \lim\limits_{m\to\infty}\lim\limits_{r\to\infty}\!
\left|\int_{\R^n}\!\frac{\xi}{|\xi|}\cdot
F(\Phi_mI_r(\cdot,p'))(\xi) \overline{F(\Phi_mU_r(\cdot,p'))(\xi)}\psi\left(\frac{\xi}{|\xi|}\right)d\xi-\right. \nonumber\\
\left.\int_{\R^n}\frac{\xi}{|\xi|}\cdot F(\Phi_m\tilde I_r(\cdot,p'))(\xi)\overline{F(\Phi_m
U_r(\cdot,p'))(\xi)}\psi\left(\frac{\xi}{|\xi|}\right)d\xi\right|=0.
\end{eqnarray}
Observe that the function $\tilde f(\lambda)=f(x,\lambda)\in C(\R,H_+^\bot)$ is continuous and does not depend on $y$.
Therefore for any $\varepsilon>0$  there exists a piece-wise constant vector-valued
function $g(\lambda)$ of the form
$g(\lambda)=\sum\limits_{i=1}^kv_i\theta(\lambda-p_i),$
 where $v_i\in H_+^\bot$, $p=p_1<p_2<\cdots<p_k=q$  such
that $\|\tilde f\chi-g\|_\infty\le\varepsilon$ on $\R$. Here $\chi(\lambda)=\theta(\lambda-p)-\theta(\lambda-q)$. Moreover, by the density of $D$, we may suppose that $p_i\in D$ for $i>1$.
We define for $p'\in D\cap (p,p_2)$
$$
J_r(y,p')=\int g(\lambda)\theta(\lambda-p')d\gamma_y^r(\lambda).
$$

Using again Plancherel's identity and the fact that
\begin{eqnarray*}
|\tilde I_r(y,p')-J_r(y,p')|=
\left|\int(\tilde f\cdot\chi-g)(\lambda)\theta(\lambda-p')d\gamma_y^r(\lambda)\right|\le \\ \int|(\tilde
f\cdot\chi-g)(\lambda)|d |\gamma_y^r|(\lambda)\le 2\varepsilon,
\end{eqnarray*}
we obtain
\begin{eqnarray}\label{20}
\left|\int_{\R^n}\frac{\xi}{|\xi|}\cdot F(\Phi_m\tilde I_r(\cdot,p'))(\xi)\overline{F(\Phi_m U_r(\cdot,p'))(\xi)}\psi
\left(\frac{\xi}{|\xi|}\right)d\xi\right.-\nonumber\\
\left.\int_{\R^n}\frac{\xi}{|\xi|}\cdot F(\Phi_m J_r(\cdot,p'))(\xi)\overline{F(\Phi_m U_r(\cdot,p'))(\xi)}\psi
\left(\frac{\xi}{|\xi|}\right)d\xi\right|=\nonumber\\
\left|\int_{\R^n}\frac{\xi}{|\xi|}\cdot F(\Phi_m(\tilde I_r(\cdot,p')-J_r(\cdot,p')))(\xi)\overline{F(\Phi_m U_r(\cdot,p'))(\xi)}\psi
\left(\frac{\xi}{|\xi|}\right)d\xi\right|\le \nonumber\\
\|\Phi_m(\tilde I_r(\cdot,p')-J_r(\cdot,p'))\|_2
\cdot\|\Phi_mU_r(\cdot,p')\|_2\cdot\|\psi\|_\infty\le 2\|\psi\|_\infty\varepsilon
\end{eqnarray}
 for all $\psi(\xi)\in C(S)$.
Since
$$
J_r(y,p')=\int\left(\sum\limits_{i=1}^k
v_i\theta(\lambda-p_i')\right)
d\gamma_y^r(\lambda)=\sum\limits_{i=1}^k v_iU_r(y,p_i'),
$$
where $p_i'=\max(p_i,p')\in D$, it follows from (\ref{repr}) with account of Remark~\ref{rem2} that
\begin{eqnarray}
\label{21} \lim_{p'\to p+}\lim_{m\to\infty} \lim_{r\to\infty}
\int_{\R^n}\frac{\xi} {|\xi|}\cdot
F(\Phi_m J_r(\cdot,p'))(\xi)\overline{F(\Phi_m
U_r(\cdot,p')(\xi)}\psi\left(\frac{\xi}{|\xi|}\right)d\xi=
\nonumber\\
\lim_{p'\to p+}\sum\limits_{i=1}^k\langle\mu_x^{p_i'p'},(v_i\cdot\xi)
\psi(\xi)\rangle=
\sum\limits_{i=1}^k\langle\mu_x^{p_ip+},(v_i\cdot\xi)
\psi(\xi)\rangle=0.
\end{eqnarray}
The  last  equality is a consequence of the inclusion
$\supp\mu_x^{p_ip+}\subset\supp\mu_x^{pp+}\subset H_+$ (because of Corollary~\ref{cor1})
 combined with the relation $v_i\bot H_+$. By (\ref{red}), (\ref{20}) and
(\ref{21}), we have
  $$
\limsup_{p'\to p+}\limsup\limits_{m\to\infty}\limsup\limits_{r\to\infty}
\left|\int_{\R^n}\frac{\xi}{|\xi|}\cdot F(\Phi_mI_r(\cdot,p'))(\xi)\overline{F(\Phi_m U_r(\cdot,p'))(\xi)}\psi
\left(\frac{\xi}{|\xi|}\right)d\xi\right|\le\const\cdot\varepsilon,
  $$
  and  it  suffices  to observe that $\varepsilon>0$
 can be arbitrary to  complete  the proof of (\ref{lim+}). The proof of relation (\ref{lim-}) is similar to the proof of (\ref{lim+}) and is omitted.
 \end{proof}

Now we assume that the sequence $u_k$ satisfies constraints (\ref{2}). We choose a subsequence $u_r$ and the corresponding H-measure $\mu^{pq}=\mu^{pq}_xdx$. Assume that $x\in\Omega''=\Omega'\cap\Omega_\varphi$, $p_0\in\R$.
As above, let $H_+$, $H_-$ be the minimal linear subspaces of $\R^n$ containing $\supp\mu^{p_0p_0+}_x$, $\supp\mu^{p_0p_0-}_x$, respectively.

\begin{theorem}[localization principle]\label{th3}
There exists a positive $\delta$ such that $(\varphi(x,\lambda)-\varphi(x,p))\cdot\xi=0$ for all $\xi\in H_+$, $\lambda\in [p_0,p_0+\delta]$ and all $\xi\in H_-$, $\lambda\in [p_0-\delta,p_0]$.
\end{theorem}

\begin{proof}
The proof is analogous to the proof of \cite[Theorem 4]{PaARMA} (if $d=2$), for arbitrary $d>1$ see the proof of \cite[Theorem~4]{PaJMS} (where the more general case of ultra-parabolic constraints was treated). For completeness we provide the details. Observe firstly that in view of (\ref{2}) the sequence of distributions
\begin{equation} \label{22a}
 \L_p^r(y)=\div_y
\left(\int\theta(\lambda-p)(\varphi(y,\lambda)-\varphi(y,p))
d\gamma_y^r(\lambda)\right) \mathop{\to}_{r\to\infty} 0 \mbox{ in }
W_{d,loc}^{-1}(\Omega).
\end{equation}
For $p,q\in D$, $q>p>p_0$ we consider the sequence of distributions
$$
\L_q^r-\L_p^r=\div_y(Q_r^p(y)),\quad r\in \N,
$$
where the vector-valued functions $Q_r^p(y)$ (~for fixed $q\in D$~) are as follows:
\begin{eqnarray}
\label{23}
Q_r^p(y)=\int(\varphi(y,\lambda)-\varphi(y,q))\theta(\lambda-q)
d\gamma_y^r(\lambda)-\nonumber \\
\int(\varphi(y,\lambda)-\varphi(y,p))\theta(\lambda-p)d
\gamma_y^r(\lambda)=\nonumber \\
\int(\varphi(y,q)-\varphi(y,\lambda))\chi(\lambda)d\gamma_y^r
(\lambda)- \nonumber\\
\int(\varphi(y,q)-\varphi(y,p))\theta(\lambda-p)d\gamma_y^r(\lambda)=\nonumber\\
\int(\varphi(y,q)-\varphi(y,\lambda))\chi(\lambda)d\gamma_y^r(\lambda)-
(\varphi(y,q)-\varphi(y,p))U_r(y,p);
\end{eqnarray}
here $\chi(\lambda)=\theta(\lambda-p)-\theta(\lambda-q)$ is the indicator function of the segment $(p,q]$.
As was already noted, $\displaystyle
\div_y(Q_r^p(y))\mathop{\to}\limits_{r\to\infty} 0$ in
$W_{d,loc}^{-1}(\Omega)$ and if $\Phi(y)\in C_0^\infty(\Omega)$ then
\begin{equation}
\label{24} \div_y(Q_r^p\Phi(y))\mathop{\to}\limits_{r\to\infty} 0 \
\mbox{ in } W_d^{-1}.
\end{equation}
Using the Fourier transformation, from (\ref{24}) we obtain
\begin{equation}
\label{25} |\xi|^{-1}\xi\cdot F(Q_r^p\Phi)(\xi)=F(g_r), \quad g_r\mathop{\to}_{r\to\infty} 0 \ \mbox{ in }
L^d(\R^n)
\end{equation}
 (see \cite{PaARMA,PaJMS} for details).

Let $\psi(\xi)\in C^\infty(S)$. By the known Marcinkiewicz multiplier theorem (cf.
\cite[Chapter 4]{St}) $\psi(\xi/|\xi|)$ is a Fourier multiplier in $L^s$ for all $s>1$.
This implies that
\begin{equation}\label{25a}
\overline{F(U_r(\cdot,p)\Phi)(\xi)}\psi\left(\frac{\xi}{|\xi|}\right)=\overline{F(h_r)}(\xi),
\end{equation}
where the sequence $h_r$ is bounded in $L^{d'}$, $d'=d/(d-1)$.

By (\ref{25}), (\ref{25a}) we obtain
$$
\int_{\R^n} |\xi|^{-1}\xi\cdot F(Q_r^p\Phi)(\xi)
\overline{F(U_r(\cdot,p)\Phi)(\xi)}\psi\left(\frac{\xi}{|\xi|}\right)=\int_{\R^n}g_r(x)\overline{h_r(x)}dx\to
0
$$
as $r\to\infty$, or in view of (\ref{23}),
\begin{eqnarray}
\label{27} \lim\limits_{r\to\infty}\left\{\int_{\R^n}
|\xi|^{-1}\xi\cdot F(U(\cdot,p)f\Phi)(\xi)
\overline{F(U_r(\cdot,p)\Phi)(\xi)}\psi\left(\frac{\xi}{|\xi|}\right)
d\xi-\right.
\nonumber\\
\left.\int_{\R^n}|\xi|^{-1}\xi\cdot F(V_r(\cdot,p)\Phi)(\xi)
\overline{F(U_r(\cdot,p)\Phi)(\xi)}\psi\left(\frac{\xi}{|\xi|}\right)
d\xi\right\}=0,
\end{eqnarray}
where
$$
f(y)=\varphi(y,q)-\varphi(y,p) \ \mbox{ and } \
V_r(y,p)=\int(\varphi(y,q)-\varphi(y,\lambda))\chi(\lambda)
d\gamma_y^r(\lambda).
$$
Obviously, (\ref{27}) remains valid for merely continuous $\psi(\xi)$.
We set in (\ref{27}) \ $\Phi(y)=\Phi_m(x-y)$ , where the functions
$\Phi_m$ were defined in section~\ref{sec1}, and pass to the limit as $m\to\infty$, $p\to p_0+$.
By (\ref{repr}) with $\Phi(y)=\varphi(y,q)-\varphi(y,p)$ and Lemma~\ref{lem1}, we obtain

\begin{eqnarray*}
\lim_{p\to p_0+}\lim\limits_{m\to\infty} \lim\limits_{r\to\infty}\int_{\R^n}
|\xi|^{-1}\xi\cdot F(U_r(\cdot,p)f\Phi_m)(\xi)
\overline{F(U_r(\cdot,p)\Phi_m)(\xi)}\psi\left(\frac{\xi}{|\xi|}\right)d\xi=
\\
\lim_{p\to p_0+}(\varphi(x,q)-\varphi(x,p))\cdot\langle\mu_x^{pp},\xi\psi(\xi)
\rangle = (\varphi(x,q)-\varphi(x,p_0))\cdot\langle\mu_x^{p_0p_0+},\xi\psi(\xi)
\rangle,
\end{eqnarray*} therefore
\begin{eqnarray}
\label{28}
(\varphi(x,q)-\varphi(x,p_0))\cdot\langle\mu_x^{p_0p_0+},\xi\psi(\xi)
\rangle= \nonumber
\\
\lim_{p\to p_0+}\lim\limits_{m\to\infty} \lim\limits_{r\to\infty}\int_{\R^n}
 |\xi|^{-1}\xi\cdot F(V_r(\cdot,p)\Phi_m)(\xi)
\overline{F(U_r(\cdot,p)\Phi_m)(\xi)}\psi
\left(\frac{\xi}{|\xi|}\right)d\xi.
\end{eqnarray}
Let  $\pi_1$ and $\pi_2$ be the orthogonal projections of $\R^n$
 onto the subspaces $H_+$  and
$H_+^\bot$, respectively; let
$\tilde\varphi(x,\lambda)=\pi_1(\varphi(x,\lambda))$,
$\bar\varphi(x,\lambda)=\pi_2(\varphi(x,\lambda))$.
 Recall that $H_+$ is the smallest subspace containing
$\supp\mu_x^{p_0p_0+}$. This readily implies that $\langle\mu_x^{p_0p_0+},\xi\psi(\xi)\rangle\in H_+$. Hence
\begin{equation}
\label{29} (\varphi(x,q)-\varphi(x,p_0))\cdot
\langle\mu_x^{p_0p_0+},\xi\psi(\xi) \rangle=
(\tilde\varphi(x,q)-\tilde\varphi(x,p_0))\cdot
\langle\mu_x^{p_0p_0+},\xi\psi(\xi)\rangle.
\end{equation}
Further, $V_r(y,p)=\pi_1(V_r(y,p))+\pi_2(V_r(y,p))$ and
\begin{eqnarray*}
\pi_1(V_r(y,p))=\int\left(\tilde\varphi(y,q)-
\tilde\varphi(y,\lambda)\right)\chi(\lambda)d\gamma_y^r(\lambda), \\
\pi_2(V_r(y,p))=\int\left(\bar\varphi(y,q)-
\bar\varphi(y,\lambda)\right)\chi(\lambda)d\gamma_y^r(\lambda).
\end{eqnarray*}
Observe that
$$
\pi_2(V_r(y,p)) =I_r(y,p),
$$
where the function $I_r(y,p)$ is defined in (\ref{Ir}) (with $p'$ replaced by $p$) for a vector-function $f(y,\lambda)=\bar\varphi(y,q)-
\bar\varphi(y,\lambda)\in H_+^\bot$.
By Proposition~\ref{pro4} we obtain
\begin{equation}
\label{30}
\lim_{p\to p_0+} \lim\limits_{m\to\infty} \lim\limits_{r\to\infty}\int_{\R^n}
|\xi|^{-1}\xi\cdot F(\pi_2(V_r(y,p))\Phi_m)(\xi)
\overline{F(U_r(\cdot,p)\Phi_m)(\xi)}\psi
\left(\frac{\xi}{|\xi|}\right)d\xi=0.
\end{equation}
Let $\tilde V_r(y,p)=\pi_1(V_r(y,p))$.
 From (\ref{28}), in view of (\ref{29}) and (\ref{30}), we see that
\begin{eqnarray*}
(\tilde{\varphi}(x,q)-\tilde{\varphi}(x,p_0))\cdot\langle\mu_x^{p_0p_0+},
\xi\psi(\xi) \rangle= \\
\lim_{p\to p_0+}\lim\limits_{m\to\infty} \lim\limits_{r\to\infty}\int_{\R^n}
|\xi|^{-1}\xi\cdot F(\tilde V_r(\cdot,p)\Phi_m)(\xi)
\overline{F(U_r(\cdot,p)\Phi_m)(\xi)}\psi
\left(\frac{\xi}{|\xi|}\right)d\xi,
\end{eqnarray*}
which in turn, by Bunyakovskii inequality
 and Plancherel's equality, gives us
the estimate
\begin{eqnarray}
\label{31} \left|(\tilde\varphi(x,q)-
\tilde\varphi(x,p_0))\cdot\langle\mu_x^{p_0p_0+}, \xi\psi(\xi)
\rangle\right|\le\nonumber\\ \limsup_{p\to p_0+}\limsup\limits_{m\to\infty}
\limsup\limits_{r\to\infty} \|\tilde
V_r(\cdot,p)\Phi_m\|_2\cdot\|U_r(\cdot,p)\Phi_m\|_2\cdot\|\psi\|_\infty\le\nonumber\\
\limsup_{p\to p_0+}\limsup\limits_{m\to\infty}
\limsup\limits_{r\to\infty}\|\tilde
V_r(\cdot,p)\Phi_m\|_2\cdot\|\psi\|_\infty.
\end{eqnarray}
Next, for $\displaystyle M_q(y)=\max\limits_{\lambda\in [p_0,q]} |\tilde\varphi(y,q)- \tilde\varphi(y,\lambda)|$
\begin{eqnarray*}
|\tilde V_r(y,p)|\le M_q(y)
\left|\int\chi(\lambda)d\left(\nu_y^r(\lambda)+\nu_y^0(\lambda)
\right)\right|= \\
M_q(y)(u_r(y,p)-u_r(y,q)+u_0(y,p)-u_0(y,q)).
\end{eqnarray*}
In view of the elementary inequality $(a+b)^2 \le 2(a^2+b^2)$
 and the relation $0\le u_r(y,p)-u_r(y,q)\le 1$,
$r\in\N\cup\{0\}$, we have
\begin{eqnarray}
\label{32} \|\tilde V_r(\cdot,p)\Phi_m\|_2^2\le 2\int_\Omega
(M_q(y))^2\bigl(
(u_r(y,p)-u_r(y,q))^2+\nonumber\\ (u_0(y,p)-u_0(y,q))^2\bigr)K_m(x-y)dy\le\nonumber\\
2\int_\Omega (M_q(y))^2 (u_r(y,p)-u_r(y,q)+\nonumber\\
u_0(y,p)-u_0(y,q)) K_m(x-y)dy.
\end{eqnarray}
 Since $p,q\in D\subset E,$  then
$$
u_r(y,p)-u_r(y,q)\rightharpoonup u_0(y,p)-u_0(y,q)
$$
as  $r\to\infty$  in  the  weak-$*$ topology
 of $L^\infty(\Omega)$ and from (\ref{32}) we  now  obtain  the
estimate
$$
\limsup\limits_{r\to\infty} \|\tilde V_r(\cdot,p)\Phi_m\|_2^2\le
4\int_\Omega (M_q(y))^2 (u_0(y,p)-u_0(y,q))K_m(x-y)dy,
$$
from which, passing to the limit as $m\to\infty,$ we obtain
\begin{equation}\label{33}
\limsup\limits_{m\to\infty} \limsup\limits_{r\to\infty} \|\tilde
V_r(\cdot,p)\Phi_m\|_2^2\le 4(M_q(x))^2(u_0(x,p_0)-u_0(x,p)).
\end{equation}
Here  we  bear  in mind that by the definition of $\Omega'$ (see, for instance, \cite[Proposition~3]{PaARMA}) $x$ is a Lebesgue point of the
functions $u_0(y,p_0)$,  $u_0(y,p)$. It is also used that $x\in \Omega_\varphi$ is a Lebesgue point of the function $(M_q(y))^2$ as well (~this
easily follows from the fact that $x$ is a
Lebesgue point of the maps $y\to\varphi(y,\cdot)$, $y\to
|\varphi(y,\cdot)|^2$ into the spaces $C(\R,\R^n)$, $C(\R)$,
respectively~). From (\ref{33}) in the limit as $p\to p_0$ it follows that
\begin{equation}\label{33a}
\limsup_{p\to p_0}\limsup_{m\to\infty}\limsup_{r\to\infty} \|\tilde
V_r(\cdot,p)\Phi_m\|_2^2\le 4(M_q(x))^2(u_0(x,p_0)-u_0(x,q)).
\end{equation}

In view of
(\ref{31}) and (\ref{33a}),
\begin{eqnarray}
\label{34} |(\tilde\varphi(x,q)-\tilde\varphi(x,p_0))\cdot
\langle\mu_x^{p_0p_0+},\xi\psi(\xi)\rangle|\le 2\|\psi\|_\infty M_q(x)\omega(q), \\
\nonumber \omega(q)=(u_0(x,p_0)-u_0(x,q))^{1/2}=(\nu_x(p_0,q])^{1/2}
\mathop{\to}\limits_{q\to p_0} 0.
\end{eqnarray}
It is clear that the set of vectors
of the form $\langle\mu_x^{p_0p_0+},\xi\psi(\xi)\rangle$,
with real $\psi(\xi)\in C(S)$
 spans the subspace $H_+$.
Hence  we  can choose functions $\psi_i(\xi)\in C(S),$
$i=1,\ldots,l$ such that the vectors
$v_i=\langle\mu_x^{p_0p_0+},\xi\psi_i(\xi)\rangle$   make  up  an
algebraic basis in $H_+$.

By (\ref{34}), for $\psi(\xi)=\psi_i(\xi)$, $i=1,\ldots,l$,  we
obtain
$$
|(\tilde\varphi(x,q)-\tilde\varphi(x,p_0))\cdot v_i |\le
c_i\omega(q)M_q(x), \quad c_i=\const,
$$
and since $v_i$,  $i=1,\ldots,l$ is a basis in $H_+$, these estimates
show that
\begin{eqnarray}
\label{35} |\tilde\varphi(x,q)-\tilde\varphi(x,p_0)|\le c\omega(q)
M_q(x)=\nonumber\\ c\omega(q)\max\limits_{\lambda\in [p_0,q]}
|\tilde\varphi(x,q)-\tilde\varphi(x,\lambda)|, \quad c=\const.
\end{eqnarray}
We take $q=p_0+\delta$, where $\delta>0$ is so small that
$2c\omega(q)=\varepsilon<1$. Then, in view of (\ref{35}),
\begin{equation}
\label{36} |\tilde\varphi(x,q)-\tilde\varphi(x,p_0)|\le
\frac{\varepsilon}{2} \max\limits_{\lambda\in [p_0,p]}
|\tilde\varphi(x,q)-\tilde\varphi(x,\lambda)|,
\end{equation}
and since $\varphi(x,q)$  is continuous with respect to $q$ and the
set $D$ is dense, the estimate (\ref{36}) holds for all
$q\in [p_0,p_0+\delta]$.

 We claim that now  $\tilde\varphi(x,p)=\tilde\varphi(x,p_0)$
  for $p\in[p_0,p_0+\delta]$. Indeed, assume that for
$p'\in [p_0,p_0+\delta]$
$$
|\tilde\varphi(x,p')-\tilde\varphi(x,p_0)|=
\max\limits_{\lambda\in [p_0,p_0+\delta]}
|\tilde\varphi(x,\lambda)-\tilde\varphi(x,p_0)|.
$$
Then for $\lambda\in [p_0,p']$  we have
\begin{eqnarray*}
|\tilde\varphi(x,p')-\tilde\varphi(x,\lambda)|\le
|\tilde\varphi(x,\lambda)-\tilde\varphi(x,p_0)|+ \\
|\tilde\varphi(x,p')-\tilde\varphi(x,p_0)|\le
2|\tilde\varphi(x,p')-\tilde\varphi(x,p_0)|
\end{eqnarray*}
 and
$$
\max\limits_{\lambda\in [p_0,p']}
|\tilde\varphi(x,p')-\tilde\varphi(x,\lambda)|\le
2|\tilde\varphi(x,p')-\tilde\varphi(x,p_0)|.
$$
We now derive from (\ref{36}) with $p=p'$  that
$$
|\tilde\varphi(x,p')-\tilde\varphi(x,p_0)|\le
\varepsilon|\tilde\varphi(x,p')-\tilde\varphi(x,p_0)|,
$$
and since $\varepsilon<1$, this implies that
$$
|\tilde\varphi(x,p')-\tilde\varphi(x,p_0)|= \max\limits_{\lambda\in
[p_0,p_0+\delta]}
|\tilde{\varphi}(x,\lambda)-\tilde{\varphi}(x,p_0)|=0.
$$
We conclude that $\varphi(x,\lambda)-\varphi(x,p_0)\in H_+^\bot$ for
all $\lambda\in [p_0,p_0+\delta]$, i.e.,
$(\varphi(x,\lambda)-\varphi(x,p_0))\cdot\xi=0$ on the
segment $[p_0,p_0+\delta]$ for all $\xi\in H_+$.

To prove that for some sufficiently small $\delta>0$
$(\varphi(x,\lambda)-\varphi(x,p_0))\cdot\xi=0$ on the
segment $[p_0-\delta,p_0]$ for all $\xi\in H_-$, we take $p,q\in D$, $q<p<p_0$ and repeat the reasonings used in the first part of the proof. As a result, we obtain the relation similar to (\ref{34})
$$
|(\tilde\varphi(x,q)-\tilde\varphi(x,p_0))\cdot
\langle\mu_x^{p_0p_0-},\xi\psi(\xi)\rangle|\le 2\|\psi\|_\infty M_q(x)\omega(q),
$$
where
\begin{eqnarray*}
M_q(x)=\max_{\lambda\in [q,p_0]} |\tilde\varphi(y,q)- \tilde\varphi(y,\lambda)|, \\
\omega(q)=\lim_{p\to p_0-}(u_0(x,q)-u_0(x,p))^{1/2}=(\nu_x(q,p_0))^{1/2}
\mathop{\to}\limits_{q\to p_0} 0.
\end{eqnarray*}
This relation readily  implies the desired statement $(\varphi(x,\lambda)-\varphi(x,p_0))\cdot\xi=0$ on the
segment $[p_0-\delta,p_0]$ for all $\xi\in H_-$, where $\delta$ is sufficiently small.

The proof is complete.
\end{proof}

\begin{corollary}\label{cor2}
Let $x\in\Omega''$, $[a,b]$ be the minimal segment, containing $\supp\nu_x$ and $p_0\in (a,b)$. Then, in the notations of Theorem~\ref{th3}, $\supp\mu^{p_0p_0+}_x\cap \supp\mu^{p_0p_0-}_x\not=\emptyset$
and for all $\xi\in H_+\cap H_-$, $\xi\not=0$ the function $\xi\cdot\varphi(x,\lambda)$ is constant in a vicinity of $p_0$.
\end{corollary}

\begin{proof}
First, note that since $x\in\Omega''\subset\Omega'$ is a Lebesgue point of the functions $u_0(\cdot,p)$ for all $p\in D$ while
$D$ is dense, the distribution function $u_0(x,\lambda)=\nu_x((\lambda,+\infty))$ is uniquely defined
by the relation $u_0(x,\lambda)=\sup\limits_{p\in D,p>\lambda} u_0(x,p)$. In particular, the measure $\nu_x$ is well-defined at the point $x$.

The statement that the function
$\lambda\to\xi\cdot\varphi(x,\lambda)$ is constant in a vicinity of $p_0$ for all $\xi\in H_+\cap H_-$, $\xi\not=0$ readily follows from the assertion of Theorem~\ref{th3}. Hence, we only need to show that  $\supp\mu^{p_0p_0+}_x\cap \supp\mu^{p_0p_0-}_x\not=\emptyset$.  We assume to the contrary that $S_+\cap S_-=\emptyset$, where
$S_\pm=\supp\mu^{p_0p_0\pm}_x$. Denote $C_+=S\setminus S_+$, $C_-=S\setminus S_-$, Then $S=C_+\cup C_-$, $\mu_x^{p_0p_0+}(C_+)=\mu_x^{p_0p_0-}(C_-)=0$. Therefore, by relation (\ref{pd2}), for all $p,q\in D$, $p<p_0<q$
\begin{eqnarray*}
\Var\mu^{pq}_x=|\mu^{pq}_x|(S)\le |\mu^{pq}_x|(C_+)+|\mu^{pq}_x|(C_-)\le \\ \left(\mu^{pp}_x(C_+)\mu^{qq}_x(C_+)\right)^{1/2}+\left(\mu^{pp}_x(C_-)\mu^{qq}_x(C_-)\right)^{1/2}\le
\left(\mu^{qq}_x(C_+)\right)^{1/2}+\left(\mu^{pp}_x(C_-)\right)^{1/2},
\end{eqnarray*}
where we use that $\mu^{pp}_x(A)\le \mu^{pp}_x(S)\le 1$ for all $p\in D$ and every Borel set $A\subset S$, see
(\ref{7}). It follows from the obtained estimate and Lemma~\ref{lem1} that
$$
\lim_{p\to p_0-}\lim_{q\to p_0+}\Var\mu^{pq}_x\le \left(\mu^{p_0p_0+}_x(C_+)\right)^{1/2}+\left(\mu^{p_0p_0-}_x(C_-)\right)^{1/2}=0.
$$
Thus,
\begin{equation}\label{van}
\mu^{pq}_x\to 0 \ \mbox{ in } \M(S) \ \mbox{ as } p\to p_0-, q\to p_0+.
\end{equation}
On the other hand, by (\ref{repr})
\begin{eqnarray}\label{37}
\mu^{pq}_x(S)=\lim_{m\to\infty}\lim_{r\to\infty} \int_{\R^n}
F(\Phi_mU_r(\cdot,p))(\xi)\overline{F(\Phi_mU_r(\cdot,q))(\xi)}
d\xi=\nonumber\\
\lim_{m\to\infty}\lim_{r\to\infty} \int_{\R^n}
U_r(y,p)U_r(y,q)K_m(x-y)dy.
\end{eqnarray}
Observe that $U_r(x,\lambda)=\theta(u_r(x)-\lambda)-u_0(x,\lambda)$.
Since $U_r(\cdot,p)\mathop{\rightharpoonup}\limits_{r\to\infty} 0$ for all $p\in D$ and $(\theta(u_r(y)-p)-1)\theta(u_r(y)-q)\equiv 0$, we find

\begin{eqnarray*}
\lim_{r\to\infty} \int_{\R^n}
U_r(y,p)U_r(y,q)K_m(x-y)dy= \\\lim_{r\to\infty} \int_{\R^n}
(U_r(y,p)-1)U_r(y,q)K_m(x-y)dy= \\
\lim_{r\to\infty} \int_{\R^n}(\theta(u_r(y)-p)-1-u_0(y,p))(\theta(u_r(y)-q)-u_0(y,q))K_m(x-y)dy= \\
\lim_{r\to\infty} \int_{\R^n}
[(1-\theta(u_r(y)-p))u_0(y,q)-u_0(y,p)(\theta(u_r(y)-q)-u_0(y,q))]K_m(x-y)dy \\ =\int_{\R^n}(1-u_0(y,p))u_0(y,q)K_m(x-y)dy.
\end{eqnarray*}
In the limit as $m\to\infty$ this yields
\begin{eqnarray*}
\lim_{m\to\infty}\lim_{r\to\infty} \int_{\R^n}
U_r(y,p)U_r(y,q)K_m(x-y)dy= \\ \lim_{m\to\infty}\int_{\R^n}(1-u_0(y,p))u_0(y,q)K_m(x-y)dy=(1-u_0(x,p))u_0(x,q).
\end{eqnarray*}
Here we take into account that $x$ is a Lebesgue point of the functions $u_0(y,p)$, $u_0(y,q)$. By (\ref{van}), (\ref{37}) we find
\begin{eqnarray*}
0=\lim_{p\to p_0-}\lim_{q\to p_0+}\mu^{pq}_x(S)=\\ \lim_{p\to p_0-}\lim_{q\to p_0+}(1-u_0(x,p))u_0(x,q)=
\nu_x((-\infty,p_0))\nu_x((p_0,+\infty))>0,
\end{eqnarray*}
since $a<p_0<b$ and $[a,b]$ is the minimal segment containing $\supp\nu_x$. The obtained contradiction implies that
$S_+\cap S_-\not=\emptyset$ and completes the proof.
\end{proof}

Now we are ready to prove Theorem~\ref{th1}.
\begin{proof}
Let $u_r=u_{k_r}$ be a subsequence of $u_k$ chosen in accordance with Proposition~\ref{pro2}. In particular, this subsequence converges to a measure-valued function $\nu_x\in\MV(\Omega)$. In view of (\ref{pr2}) for a.e. $x\in\Omega$
\begin{equation}\label{mean}
u(x)=\int\lambda d\nu_x(\lambda).
\end{equation}
We define the set of full measure $\Omega''\subset\Omega$ and the minimal segment $[a(x),b(x)]$, containing $\supp\nu_x$, $x\in\Omega''$. In view of (\ref{mean}) $u(x)\in (a(x),b(x))$ whenever $a(x)<b(x)$. By Corollary~\ref{cor2} the function $\xi\cdot\varphi(x,\cdot)$ is constant in a vicinity of $u(x)$ for some vector $\xi\not=0$. But this contradicts to the assumption of Theorem~\ref{th1}. Therefore, $a(x)=b(x)=u(x)$ for a.e. $x\in\Omega$. This means that $\nu_x(\lambda)=\delta(\lambda-u(x))$. By Theorem~\ref{thT} the subsequence $u_r\to u$ as $r\to\infty$ in $L^1_{loc}(\Omega)$. Finally, since the limit function $u(x)$ does not depend of the choice of a subsequence $u_r$, we conclude that the original sequence $u_k\to u$ in $L^1_{loc}(\Omega)$ as $k\to\infty$. The proof is complete.
\end{proof}

\section{Decay property}
This section is devoted to the proof of Theorem~\ref{th2}. Suppose that $u(t,x)$ is a unique e.s. to problem
(\ref{cl}), (\ref{ini}) with the periodic initial data $u_0(x)$. By Remark~\ref{rem1} we can assume that
$u(t,x)\in C([0,+\infty),L^1(\T^n))$ (after possible correction on a set of null measure). We consider the sequence
$u_k(t,x)=u(kt,kx)$, $k\in\N$, consisting of e.s. of (\ref{cl}). As was firstly shown in \cite{ChF}, the decay property (\ref{dec}) is equivalent to the strong convergence $u_r(t,x)\mathop{\to}\limits_{r\to\infty} I=\const$ in $L^1_{loc}(\Pi)$ of a subsequence $u_r=u_{k_r}(t,x)$. As follows from \cite[Lemma~3.2(i)]{PaAIHP},
$u_r\rightharpoonup u^*$, where $u^*=u^*(t)$ is a weak-$*$ limit of the sequence $a_0(k_rt)$, where
$\displaystyle a_0(t)=\int_{\T^n}u(t,x)dx$. Since $u(t,x)$ is an e.s. of (\ref{cl}), this function is constant:
$\displaystyle a_0(t)\equiv I=\int_{\T^n}u_0(x)dx$, in view of (\ref{mass}). Therefore, $u_r\rightharpoonup I$ as $r\to\infty$
(actually, the original sequence $u_k\rightharpoonup I$ as $k\to\infty$).

Let $\mu^{pq}$, $p,q\in E$, be the H-measure corresponding to a subsequence $u_r=u_{k_r}(t,x)$.
Recall that $\mu^{pq}=\mu^{pq}(t,x,\tau,\xi)\in\M_{loc}(\Pi\times S)$, where
$$
S=\{ \ \hat\xi=(\tau,\xi)\in\R\times\R^n \ | \ |\hat\xi|^2=\tau^2+|\xi|^2=1 \ \}
$$
is a unit sphere in the dual space $\R^{n+1}$ (the variable $\tau$ corresponds to the time variable $t$).

By \cite[Theorem~3.1]{PaAIHP} the following localization principle holds $$\supp\mu^{pq}\subset\Pi\times S_0,$$ where
$$
S_0= \{ \ \hat\xi/|\hat\xi| \ | \ \hat\xi=(\tau,\xi)\not=0, \tau\in\R, \ \xi\in L' \ \}.
$$
As was demonstrated in Proposition~\ref{pro3}, $\mu^{pq}=\mu^{pq}_{t,x}dtdx$ for all $p,q\in D$, where
$D\subset E$ is a countable dense subset and measures $\mu^{pq}_{t,x}\in \M(S)$, are defined for all $(t,x)$ belonging to a set of full measure $\Pi'\subset\Pi$. Obviously, the identity
\begin{equation}\label{38}
\langle\mu^{pp},\Phi(t,x,\hat\xi)\rangle =\int_\Pi
\langle\mu^{pp}_{t,x}(\hat\xi),\Phi(t,x,\hat\xi)\rangle dtdx,
\end{equation}
$\Phi(t,x,\hat\xi)\in C_0(\Pi\times S)$, remains valid also for compactly supported Borel functions $\Phi$. Taking $\Phi=\phi(t,x)h(\hat\xi)$, where $\phi(t,x)\in C_0(\Pi)$, $\phi(t,x)\ge 0$ while $h(\hat\xi)$ is an indicator function of the set $S\setminus S_0$, we derive from (\ref{38}) that
$$
\int_\Pi
\mu^{pp}_{t,x}(S\setminus S_0)\phi(t,x)dtdx=0
$$
and since $\mu^{pp}_{t,x}\ge 0$ and $\phi(t,x)\in C_0(\Pi)$ is arbitrary nonnegative function, it follows from this identity that $\mu^{pp}_{t,x}(S\setminus S_0)=0$ for all $p\in D$, $(t,x)\in\Pi'$. By relation (\ref{pd2}) we claim that, more generally,  $|\mu^{pq}_{t,x}|(S\setminus S_0)=0$ for all $p,q\in D$, $(t,x)\in\Pi'$. Finally, in view of Lemma~\ref{lem1}, we find that $|\mu^{pq\pm}_{t,x}|(S\setminus S_0)=0$, that is,
\begin{equation}\label{39}
\supp\mu^{pq\pm}_{t,x}\subset S_0 \ \forall p,q\in\R, (t,x)\in\Pi'.
\end{equation}
Further, $u_r(t,x)$ is a sequence of entropy solutions of (\ref{cl}). Therefore (~see for instance \cite{PaJMS}~) the sequences
$$
\div[\theta(u_r-p)(\hat\varphi(u_r)-\hat\varphi(p))]= ((u_r-p)^+)_t+\div_x[\theta(u_r-p)(\varphi(u_r)-\varphi(p))]
$$
are compact in $H^{-1}_{d,loc}(\Pi)$ for some $d>1$ and all $p\in\R$, where $\hat\varphi(u)=(u,\varphi(u))\in C(\R,\R^{n+1})$, and we use the notation $v^+=\max(v,0)$.

Denote by $\nu_{t,x}\in\MV(\Pi)$ the limit measure valued function for a sequence $u_r$, and by $[a(t,x),b(t,x)]$ the minimal segment containing $\supp\nu_{t,x}$.

Suppose that $(t,x)\in\Pi'$, $a(t,x)<b(t,x)$. Then $I=\int\lambda d\nu_{t,x}(\lambda)\in (a(t,x),b(t,x))$. By Corollary~\ref{cor2} we find that there exists $\hat\xi=(\tau,\xi)\in\supp\mu^{II+}_{t,x}\cap\supp\mu^{II-}_{t,x}$ and $\delta>0$ such that the function
\begin{equation}\label{40}
\lambda\to\hat\xi\cdot\hat\varphi(\lambda)=\tau u+\xi\cdot\varphi(u)=c=\const
\end{equation}
on the interval $V=\{\lambda \ | \ |\lambda-I|<\delta\}$. By (\ref{39}) $\hat\xi\in S_0$, which implies that we can assume that $\xi\in L'$ in (\ref{40}). Evidently, $\xi\not=0$ (otherwise, $\tau u\equiv c$ on $V$ for $\tau\not=0$).
Hence the function $\xi\cdot\varphi(u)=c-\tau u$ is affine, which contradicts (\ref{ND2}). Thus, $a(t,x)=b(t,x)=I$ for a.e. $(t,x)\in\Pi$. We conclude that $\nu_{t,x}(\lambda)=\delta(\lambda-I)$ an by Theorem~\ref{thT} the sequence
$u_r\to I$ as $r\to\infty$ strongly (in $L^1_{loc}(\Pi)$~). As was mentioned above (one can simply repeat the conclusive part of the proof of Theorem~1.1 in \cite{PaAIHP}), this implies (\ref{dec}).

Conversely, if the assumption (\ref{ND2}) fails, we can find $\xi\in L'$, $\xi\not=0$, and constants $a,b\in\R$ such that
$\xi\cdot\varphi(\lambda)\equiv au+b$ on a segment $[I-\delta,I+\delta]$, $\delta>0$. Then, as is easily verified,
the function
$$
u(t,x)=I+\delta\sin(2\pi(\xi\cdot x-at))
$$
is the e.s. of (\ref{cl}), (\ref{ini}) with initial data $u_0(x)=I+\delta\sin(2\pi(\xi\cdot x))$. It is clear that $u_0(x)$ is $L$-periodic and $\int_{\T^n} u_0(x)dx=I$, but the e.s. $u(t,x)$ does not satisfy the decay property.

\begin{example}
Let $n=1$, $\varphi(u)=|u|$.
Let $u=u(t,x)$ be an e.s. of the problem
\begin{equation}\label{41}
u_t+(|u|)_x=0, \quad u(0,x)=u_0(x),
\end{equation}
where $u_0(x)\in L^\infty(\R)$ is a nonconstant periodic function with a period $l$ (for a constant $u_0\equiv c$ the e.s. $u\equiv c$ and the decay property is evident). Notice that no previous results \cite{ChF,Daferm,PaAIHP} can help to answer the question whether the decay property is satisfied. However, as follows from Theorem~\ref{th2}, if $I=\frac{1}{l}\int_0^l u_0(x)dx=0$, then the decay property holds: $\int_0^l|u(t,x)|dx\to 0$ as $t\to\infty$.
Actually, the condition $\int_0^l u_0(x)dx=0$ is also necessary for the decay property (\ref{dec}). Indeed, $u(t,x)=u_0(x\mp t)$ if $\pm u_0(x)\ge 0$ (then $\pm I>0$), and the decay property is evidently violated.
In the remaining case when $u_0$ changes sign we define the functions $u_+(t,x)=v_+(x-t)$, $u_-(t,x)=v_-(x+t)$,
where $v_+(x)=\max(u_0(x),0)\ge 0$, $v_-(x)=\min(u_0(x),0)\le 0$. Note that this functions take zero values on sets of positive measures. By the construction, $v_-(x)\le u_0(x)\le v_+(x)$ and $u_\pm(t,x)$ are e.s. of (\ref{41}) with initial data $v_\pm(x)$. In view of the known property of monotone dependence of e.s. on initial data
$u_-(t,x)\le u(t,x)\le u_+(t,x)$ a.e. on $\Pi$. These inequality can be written in the form
\begin{equation}\label{42}
u(t,x-t)\ge v_-(x), \quad u(t,x+t)\le v_+(x).
\end{equation}
Assuming that $u(t,x)$ satisfies the decay property, we find, with the help of $x$-periodicity of $u(t,\cdot)$, that
$$
\int_0^l |u(t,x\pm t)-I|dx=\int_0^l |u(t,x)-I|dx\to 0 \ \mbox{ as } t\to +\infty,
$$
that is, the functions $u(t,x\pm t)\mathop{\to}\limits_{t\to +\infty} I$ in $L^1([0,l])$.
Passing to the limit as $t\to +\infty$ in (\ref{42}), we find that $v_-(x)\le I\le v_+(x)$ for a.e. $x\in\R$. The latter is possible only if $I=0$. We conclude that the decay property holds only in the case $I=0$.
\end{example}

\begin{remark}
Theorem~\ref{th2} can be extended to more general case of almost periodic initial data (in the Besicovitch sense \cite{Bes}).
Repeating the arguments of \cite{PaAP}, we arrive at the following analogue of Theorem~\ref{th2}.

\begin{theorem}\label{th4}
Let $M_0$ be the additive subgroup of $\R^n$ generated by the spectrum of $u_0$. Assume that
for all $\xi\in M_0$, $\xi\not=0$ the function $\xi\cdot\varphi(\lambda)$ is not affine in any vicinity of
$I=\dashint_{\R^n} u_0(x)$.
Then the e.s. $u(t,x)$ of (\ref{cl}), (\ref{ini}) satisfies the decay property
$$
\lim_{t\to +\infty}\dashint_{\R^n} |u(t,x)-I|dx=0.
$$
Here $\displaystyle\dashint_{\R^n} v(x)dx$ denotes the mean value of an almost periodic function $v(x)$ (see \cite{Bes}~).
\end{theorem}
\end{remark}

{\bf Acknowledgements.}
The author is supported by Russian Foundation for Basic Research (grant 15-01-07650-a) and the Ministry of Education and Science of Russian Federation within the framework of state task (project no. 1.857.2014/K).


\begin{thebibliography}{99}
\small
\itemsep=-2pt
\bibitem{Bes}
A.\,S.~Besicovitch, Almost Periodic Functions. Cambridge University Press, 1932.
\bibitem{ChF}
G.-Q.~Chen, H.~Frid, Decay of entropy solutions of nonlinear conservation laws, Arch.
Rational Mech. Anal. 146 (2) (1999) 95--127.
\bibitem{Daferm}
C.\,M.~Dafermos, Long time behavior of periodic solutions to scalar conservation laws in several space dimensions,
SIAM J. Math. Anal. 45:4 (2013), 2064--2070.
\bibitem{Di}
R.\,J.~DiPerna,  Measure-valued solutions to conservation laws, Arch. Rational Mech.
Anal. 88 (1985) 223--270.
\bibitem{Ger}
P.~Ger\'ard, Microlocal defect measures, Comm. Partial Diff. Equat. 16 (1991) 1761--1794.
\bibitem{HF}
E.~Hille, R.S.~Phillips, Functional analysis and semi-groups. Providence, 1957.
\bibitem{Kr}
S.\,N.~Kruzhkov,  First order quasilinear equations in several independent variables,
Mat. Sb. 81 (1970) 228--255, English transl. in Math. USSR Sb.
10 (1970) 217--243.
\bibitem{KrPa1}
S.\,N.~Kruzhkov, E.\,Yu.~Panov, First-order conservative quasilinear laws with an infinite
domain of dependence on the initial data, Dokl. Akad. Nauk SSSR 314 (1990) 79--84,
English transl. in Soviet Math. Dokl. 42 (1991) 316--321.
\bibitem{KrPa2}
S.\,N. Kruzhkov, E.\,Yu.~Panov, Osgood's type conditions for uniqueness of entropy solutions
to Cauchy problem for quasilinear conservation laws of the first order, Ann. Univ. Ferrara
Sez. VII (N.S.) 40 (1994) 31--54.
\bibitem{Pa3}
E.\,Yu.~Panov, On sequences of measure-valued solutions of first-order quasilinear equations,
Mat. Sb. 185 (2) (1994) 87--106, English transl. in Russian Acad. Sci.
Sb. Math. 81 (1) (1995) 211--227.
\bibitem{Pa5}
E.\,Yu.~Panov, Property of strong precompactness for bounded sets of measure valued solutions
of a first-order quasilinear equation, Mat. Sb. 190 (3) (1999) 109--128,
English transl. in Russian Acad. Sci. Sb. Math. 190 (3) (1999) 427--446.
\bibitem{PaMax1}
E.\,Yu.~Panov, A remark on the  theory of generalized entropy sub- and supersolutions of the
Cauchy problem for  a  first-order quasilinear equation, Differ. Uravn.
37 (2) (2001) 252--259, English transl. in Differ. Equ. 37 (2) (2001) 272--280.
\bibitem{Pa6}
E.\,Yu.~Panov, Existence of strong traces for generalized solutions of multidimensional
scalar conservation laws, J. Hyperbolic Differ. Equ. 2 (4) (2005) 885--908.
\bibitem{Pa7}
E.\,Yu.~Panov, Existence of strong traces for quasi-solutions of multidimensional
conservation laws, J. Hyperbolic Differ. Equ. 4 (4) (2007) 729--770.
\bibitem{PaARMA}
E.\,Yu.~Panov, Existence and strong pre-compactness properties for entropy solutions of a
first-order quasilinear equation with discontinuous flux, Arch. Rational Mech.  Anal.
195 (2) (2010) 643--673.
\bibitem{PaJMS}
E.Yu.~Panov,
Ultra-parabolic equations with rough coefficients. Entropy solutions and strong
precompactness property, J. Mathematical Sciences 159 (2) (2009) 180--228.
\bibitem{PaAIHP}
E.\,Yu.~Panov, On decay of periodic entropy solutions to a scalar conservation law, Annales de l'Institut Henri Poincar\'e (C) Analyse Non Lin\'eaire 30:6 (2013), 997-–1007.
\bibitem{PaAP}
E.\,Yu.~Panov, On the Cauchy problem for scalar conservation laws in the class of Besicovitch almost periodic
functions: global well-posedness and decay property, arXiv:1406.4838v2, 19 Jun 2014.
\bibitem{St}
E.\,M.~Stein, Singular Integrals and Differentiability Properties of Functions, Princeton Univ.
Press, Princeton, N.J. (1970).
\bibitem{Ta1}
L.~Tartar, Compensated compactness and applications to partial differential equations,
Nonlinear analysis and mechanics: Heriot. Watt Symposium, vol. 4 (Edinburgh 1979), Res. Notes
Math. 39 (1979) 136--212.
\bibitem{Ta2}
L.~Tartar, H-measures,  a  new  approach for studying homogenisation, oscillations and
concentration  effects in partial differential equations, Proc. Roy. Soc. Edinburgh. Sect. A.
115 (3-4) (1990) 193--230.
\end{thebibliography}
\end{document}